\documentclass[12pt]{amsart}
\usepackage{color}
\usepackage[all]{xy}
\usepackage{amssymb,amsmath}

\usepackage{setspace}

\date{August 26, 2017}

\calclayout
\newtheorem{dummy}{anything}[section]
\newtheorem{theorem}[dummy]{Theorem}
\newtheorem*{thma}{Theorem A}
\newtheorem*{thmb}{Theorem B}

\newtheorem{lemma}[dummy]{Lemma}
\newtheorem{proposition}[dummy]{Proposition}
\newtheorem{corollary}[dummy]{Corollary}

\theoremstyle{definition}%%Change Theoremstyle
\newtheorem{definition}[dummy]{Definition}

  \newtheorem{remark}[dummy]{Remark}

  \newtheorem*{acknowledgement}{Acknowledgement}

\newcommand
{\eqncount}{\setcounter{equation}{\value{dummy}}%
\addtocounter{dummy}{1}}
\newcommand{\cB}{\mathcal B}

\newcommand{\cF}{\mathcal F}

\newcommand{\bZ}{\mathbb Z}

\newcommand{\wX}{\widetilde X}

\DeclareMathOperator{\Image}{im}

\newcommand{\bd}{\partial}
\newcommand{\id}{\mathrm{id}}

\newcommand{\La}{\Lambda}

\newcommand{\fake}{{\textup{D2}-complex}}
\newcommand{\fakes}{{\textup{D2}-complexes}}
\newcommand{\mmatrix}[4]{\bigg (\hskip-4pt\vcenter
{\xymatrix@C-2pc@R-2pc{#1&#2\\#3&#4} }\hskip-2pt
\bigg )}

\newcommand{\ad}{\hphantom{-}}
\newcommand{\ZG}{\bZ G}

\DeclareMathOperator{\Def}{def}

\begin{document}

\title{Two remarks on Wall's D2 problem}
\author{Ian Hambleton}

\address{Department of Mathematics \& Statistics, McMaster University,   Ontario L8S 4K1, Canada}

\email{hambleton@mcmaster.ca }

\thanks{Research partially supported by NSERC Discovery Grant A4000.}

\begin{abstract} If a finite group $G$ is isomorphic to a subgroup of $SO(3)$, then $G$ has the D2-property. 
Let $X$ be a finite complex satisfying Wall's D2-conditions. If $\pi_1(X)=G$ is finite, and $\chi(X) \geq 1-\Def(G)$,  then $X \vee S^2$ is simple homotopy equivalent to a finite $2$-complex, whose simple homotopy type depends only on $G$ and $\chi(X)$.
\end{abstract}

\maketitle

\section{Introduction}
\label{sect: introduction}
In \cite[\S 2]{wall-finiteness1}, C.~T.~C.~Wall initiated the study of  the relations between homological and geometrical dimension conditions for finite $CW$-complexes. In particular, a finite complex $X$ \emph{satisfies Wall's \textup{D2}-conditions} if  $H_i(\wX) =0$, for $i >2$, and $H^{3}(X; \cB) = 0$, for all coefficient bundles $\cB$. Here $\wX$ denotes the universal covering of $X$. If these conditions hold, we will say that $X$ is a
\emph{\fake}. If every \fake\ with fundamental group $G$ is homotopy equivalent to a finite $2$-complex, then we say that \emph{$G$ has the \textup{D2}-property}.

In \cite[p.~64]{wall-finiteness1}, Wall proved that a finite complex $X$ satisfying the D2-conditions is homotopy equivalent to a finite $3$-complex. We will therefore assume that all our D2-complexes have $\dim X \leq 3$.

	The D2 problem for a finitely-presented group $G$ asks whether every finite complex $X$ with fundamental group $G$ which satisfies the D2-conditions is homotopy equivalent to a finite $2$-complex. The D2 problem has been actively studied for finite groups, but answered affirmatively only in a limited number of cases  (see  \cite{hog-angeloni-metzler1},  \cite{johnson2} for references to the literature on $2$-complexes and the D2-problem, and  compare \cite{Mannan:2009a}, \cite{Jin:2015}, 
\cite{ye:2015}  for some more recent work).

In this note, I make two  remarks concerning the (stable) solution of the D2-problem and cancellation, based on my joint work with Matthias Kreck \cite[Theorem B]{hk4}. I am indebted to Dr.~W.~H.~Mannan for asking about this connection some years ago. 

For $G$ a finitely presented group, let $\Def(G)$ denote the \emph{deficiency of $G$}, defined as the maximum value of the number of generators minus the number of relations over all finite presentations of $G$.  We note that $1-\Def(G)$ is the minimal Euler characteristic possible for a finite $2$-complex with fundamental group $G$.

Swan defined $\mu_2(G)$ as the minimum of the numbers $\mu_2(\cF) =f_2 - f_1 + f_0$, where $f_i$ are the ranks of the finitely generated free $\ZG$-modules $F_i$ in an exact sequence
$$ F_2 \to F_1 \to F_0 \to \bZ \to 0.$$
In general, Swan \cite[Proposition 1]{Swan:1965} noted that $\mu_2(G) \leq 1-\Def(G)$. For a finite D2-complex X, we have the Euler characteristic inequality $\chi(X) \geq \mu_2(G)$ (see Section \ref{sec:one} for details). In addition, $\mu_2(G) \geq 1$ for $G$ a finite group by \cite[Corollary 1.3]{Swan:1965}.

\begin{thma} Let $X$ be a finite complex satisfying the \textup{D2}-conditions, and assume that $G:=\pi_1(X)$ is a finite group. Then
\begin{enumerate}
\item if $\chi(X) > 1-\Def(G) $,  $X$ is simple homotopy equivalent to a finite $2$-complex;
\item  If $\chi(X) =1-\Def(G)$, 
 $ X \vee S^2$ is simple homotopy equivalent to a finite $2$-complex.
\end{enumerate}
 In case \textup{(i)} the simple homotopy type of $X$  depends only on $\pi_1(X)$ and $\chi(X)$.
 \end{thma}

 The uniqueness part is a direct application of \cite[Theorem B]{hk4}, since the resulting $2$-complexes have non-minimal Euler characteristic.  We remark that the unpublished work of Browning \cite{Browning:1978} implies the corresponding weaker statements for homotopy equivalence, rather than simple homotopy equivalance (see Corollary \ref{cor:browning}).

\begin{remark}  A stable solution of the problem for D2-complexes with any finitely presented fundamental group was first given by Cohen \cite[Theorem 1]{Cohen:1978}:  if $X$ is  a D2-complex, then there exists an integer $r\geq 0$  such that the stabilized complex
 $ X \vee r(S^2)$ is homotopy equivalent to a finite $2$-complex.
 
This result and the foundational work of 
J.~H.~C.~Whitehead \cite{whitehead1} shows that  any two \fakes\ with isomorphic fundamental groups become stably simple homotopy equivalent after wedging on sufficiently many $2$-spheres. I give a different argument in Lemma \ref{lem:stableequiv} for the stable result,  and show that it holds whenever $r\geq b_3(X)$ (compare 
\cite[Proposition 3.5]{ye:2015}). Here $b_3(X)$ denotes the number of $3$-cells in $X$. 

If the group ring $\ZG$ is noetherian, then there exists a uniform bound for this stable range, depending only on the fundamental group (see Proposition \ref{prop:noetherian}). This remark applies for example to polycyclic-by-finite fundamental groups.
\end{remark}

\begin{thmb}
Let $G$ be a finite subgroup of $SO(3)$. Then any \textup{D2}-complex is simple homotopy equivalent to a finite $2$-complex, and $G$ has the \textup{D2}-property.
\end{thmb}

This result is an application of \cite[Theorem 2.1]{hk4}. The result was known for cyclic and dihedral groups
 (see  \cite{Mannan:2007}, \cite{OShea:2012}, \cite{Mannan:2013}), but the argument given here is more uniform and the tetrahedral, octahedral and isosahedral groups do not seem to have been covered before.

\begin{remark}Brown and Kahn \cite[Theorem 2.1]{brown-kahn1} proved that that a \fake\   which is a nilpotent space is homotopy equivalent to a $2$-complex, but this does not appear to settle the D2 problem for nilpotent fundamental groups.
\end{remark}

\begin{remark} A result essentially contained in the proof of Wall \cite[Theorem 4]{wall-finiteness2} shows that there exist  finite D2-complexes $X$, with $\pi_1(X) = G$ and $\chi(X) = \mu_2(G)$ realizing this minimum value, for every finitely presented group $G$. Since $\mu_2(G) \leq 1- \Def(G)$ by Swan \cite[Proposition 1]{Swan:1965}, a \emph{necessary} condition for any group $G$ to have the D2-property is that $\mu_2(G) = 1-\Def(G)$. 
 \end{remark}

\begin{acknowledgement} I would like to thank Jens Harlander and Jonathan Hillman for helpful comments and references.
\end{acknowledgement}
\section{Cancellation and the D2 Problem}\label{sec:one}
We  assume that $X$ is a finite, connected $3$-complex,  with fundamental group $G = \pi_1(X)$, satisfying the D2-conditions. We use the following notation for the chain complex $C(\wX;\bZ)$ of the universal covering:
$$C(X) :  0 \to C_3 \xrightarrow{\bd_3} C_2 \xrightarrow{\bd_2} C_1\xrightarrow{\bd_1} C_0 \to \bZ \to 0,$$
considered as a chain complex of finitely-generated, free $\La$-modules  relative to a single $0$-cell as base-point, where $\La = \ZG$ is the integral group ring.

The boundary map  $\bd_3$ is injective because $H_3(\wX) = 0$. Let $B_3= \Image(\bd_3)$, with $j\colon B_3 \to C_2$ the inclusion map, and consider the boundary map $\bd_3\colon C_3 \to B_3$ as defining a $3$-cocycle. Since $H^3(X; B_3) = 0$, there is a $\La$-module homomorphism $\phi\colon C_2 \to B_3$ such that $\phi\circ j = \id$. We have an exact sequence 
$$ 0 \to C_3 \to \pi_2(K) \to \pi_2(X) \to 0$$
where $K \subset X$ denotes the $2$-skeleton (since $\pi_2(K) = Z_2 = \ker \bd_2$).
It follows that 
$C_3$ is a direct summand of $\pi_2(K)$, and hence $\pi_2(X)$ is a representative of the stable class $\Omega^3(\bZ)$. More explicitly, the map $\phi$ induces a direct sum splitting $C_2 = \Image(\bd_3) \oplus P$, and  $P\cong C_2/\Image(\bd_3)$ is a finitely-generated, stably-free $\La$-module since $C_3 \cong \Image(\bd_3)$ is a finitely-generated, free $\La$-module. This gives a 
commutative diagram:
$$\xymatrix@R-4pt@C-4pt{&0\ar[d]&0\ar[d]&&\cr
&C_3\ar[d]\ar@{=}[r]^{\bd_3}&B_3\ar[d]&&\cr
0 \ar[r]&Z_2\ar[d]\ar[r]&C_2\ar[d]\ar[r]&B_2 \ar@{=}[d]\ar[r]&0\cr
0 \ar[r]&\pi_2(X)\ar[d]\ar[r]&C_2/B_3\ar[d]\ar[r]&B_2 \ar[r]&0\cr
&0&0&&}
$$
where the vertical sequences are split exact, and hence a resolution
$$0 \to \pi_2(X) \to P \to C_1 \to C_0 \to \bZ\to 0\ .$$
By a sequence of elementary expansions (on the chain complex these are just the direct sum with copies of 
$\xymatrix@C-15pt{\La \ar@{=}[r]& \La}$ in dimensions 1 and 2), we may assume that $P$ is a finitely-generated, free $\La$-module. This operation doesn't change the (simple) homotopy type of $X$. The following result has also been observed in \cite{Cohen:1978}, \cite[Theorem 3.5]{ye:2015}.
Our proof uses the techniques of \cite[\S 2]{hk4}.

\begin{lemma}\label{lem:stableequiv}
The stabilized complex $X \vee r(S^2)$, with $r=b_3(X)$,  is simple homotopy equivalent to a finite $2$-complex $K$.
\end{lemma}
\begin{proof} Let $u \colon K \subset X$ denote the inclusion of the $2$-skeleton of $X$, so that we have the identification $\pi_2(K) \cong \pi(X) \oplus C_3$  discussed above. We further identify
\eqncount
\begin{equation}\label{eq:twotwo}
 \pi_2(K \vee r(S^2)) \cong \pi_2(K) \oplus \La^r\cong \pi_2(X) \oplus C_3 \oplus F
 \end{equation}
and fix free $\La$-bases $\{e_1, \dots, e_r\}$  for $C_3 \cong \La^r$, and  $\{f_1, \dots, f_r\}$ for $F \cong \La^r$. 
The same notation $\{e_i\}$ and $\{f_j\}$ will also be used for continuous maps $S^2 \to K \vee r(S^2)$ in the homotopy classes 
of $ \pi_2(K \vee r(S^2))$ defined by these basis elements. Notice that the maps $f_j\colon S^2 \to K \vee r(S^2)$ may be chosen
to represent the inclusions of the $S^2$ wedge factors.

\smallskip
We first claim that there exists a (simple) self-homotopy equivalence
$$h\colon K \vee r(S^2) \to K \vee r(S^2)$$
such that the induced isomorphism
$$h_*\colon \pi_2(K \vee r(S^2)) \xrightarrow{\cong} \pi_2(K \vee r(S^2))$$
has the property
$h_*(e_i) = f_i$, for $1\leq i \leq r$,  with respect to the chosen bases
in the right-hand side of \eqref{eq:twotwo}, and induces the identity on the summand $\pi_2(X)$. 

The construction of the required self-homotopy equivalences is given in \cite[p.~101]{hk4}, where the realization of the group of elementary automorphisms 
$E(P_1, L\oplus P_0)$ is studied. In this notation $P_0$, $P_1$ are free modules of rank one, and $L$ is an arbitrary $\La$-module. The basic construction is to realize automorphisms of the form $1+f$ and $1 +g$, where
$f\colon L\oplus P_0 \to P_1$ and $g\colon P_1 \to L \oplus P_0$ are arbitrary $\La$-homomorphisms. 
We apply this to the sub-module $L \oplus \La\cdot e_1\oplus \La\cdot  f_1$, where $L = \pi_2(X)$, and realize the automorphism $\id_L \oplus \alpha$ with $\alpha(e_1) = -f_1$ and $\alpha(f_1) = e_1$ via the composition
$$\mmatrix{{\ad 0}}{1}{-1}{0}=\mmatrix{{\ad 1}}{0}{-1}{1}\mmatrix{{1}}{1}{0}{1}\mmatrix{{\ad 1}}{0}{-1}{1}.$$

We can now construct a homotopy equivalence $f\colon X \vee r(S^2)\to K$, by extending the simple homotopy equivalence $h \colon K \vee r(S^2) \to K \vee r(S^2)$ over the (stabilized) inclusion  
$$u \vee \id\colon K \vee r(S^2)  \to X \vee r(S^2)$$
 by attaching the $3$-cells of X  in domain, and  $3$-cells in the range which cancel the $S^2$ wedge factors.  For the attaching maps $[\bd D^3_i ]= e_i$, $1\leq i \leq r$, of the $3$-cells of $X$ we have $h \circ [\bd D^3_i ]= f_i$. Hence we can extend by the identity to $3$-cells attached along the maps $\{f_i\colon S^2 \to  K \vee r(S^2)\}$. We obtain
a map
$$h'\colon X\vee r(S^2)  \to K \vee r(S^2) \bigcup\{ D^3_i : [\bd D^3_i] = f_i, 1\leq i \leq r\} \simeq K$$
extending $h$. It is easy to check that $h'$ is a (simple) homotopy equivalence.
\end{proof}

An \emph{algebraic $2$-complex over the group ring $\La:=\ZG$} is a chain complex $(F_*, \bd_*)$ of the form
$$ F_2 \xrightarrow{\bd_2} F_1\xrightarrow{\bd_1} F_0 $$
consisting of an exact sequence of finitely-generated, stably-free $\La$-modules, such that $H_0(F_*) = \bZ$. An \emph{$r$-stabilization} of an algebraic $2$-complex is the result of direct sum with a complex $(E_*, \bd_*)$, where $E_2 = \La^r$ for some $r\geq 0$, $\bd_* = 0$ and $E_i = 0$ for $i \neq 2$. We say that an algebraic $2$-complex is \emph{geometrically realizable} if it is chain homotopy equivalent to the cellular chain complex $C(X)$ of a (geometric) finite $2$-complex  $X$ with fundamental group $\pi_1(X) =G$.
\begin{lemma}\label{lem:algreal}
Any algebraic $2$-complex $(F_*, \bd_*)$ over $\La = \ZG$ is  geometrically realizable after an $r$-stablization, for some $r\geq 0$.
\end{lemma}
\begin{proof} We compare the resolution 
$$0 \to L \to F_2 \to F_1 \to F_0 \to \bZ\to 0,$$
where $L = \ker\bd_2$,  to one obtained from the chain complex 
$$0 \to \pi_2(K) \to C_2(K) \to C_1(K) \to C_0(K) \to \bZ \to 0$$
of any finite $2$-complex $K$ with fundamental group $G$. Then Schanuel's Lemma shows that these two resolutions of $\La$-modules (regarded as connected $3$-dimensional chain complexes) are stably chain isomorphic after direct sum with elementary complexes of the form 
$ \xymatrix@C-15pt{\La \ar@{=}[r]& \La}$ in degrees $(i, i-1)$ for $1\leq i \leq 3$ (compare \cite[Lemma 3B]{wall-finiteness2}, or \cite[p.~415]{hpy1}). 

The stabilizations in degrees $(i, i-1)$ for $i < 3$ produce a complex $(F'_*, \bd'_*)$ of finitely generated free $\La$-modules, and a chain homotopy equivalence  $(F'_*, \bd'_*) \simeq  (F_*, \bd_*)$. The additional degree $(3,2)$ stabilizations produce a  complex $(F''_*, \bd''_*)$, and a  chain homotopy equivalence  $(F''_*, \bd''_*) \simeq  (F_*, \bd_*)\oplus (E_*, \bd_*)$, where $(E_*, \bd_*)$ is a complex concentrated in degree 2 (as defined above). 

In other words, the resulting stabilized complex $  (F_*, \bd_*) \oplus (E_*, \bd_*) $ is an $r$-stabilization of $(F_*, \bd_*)$.
The chain homotopy equivalence
$$  (F_*, \bd_*) \oplus (E_*, \bd_*)  \simeq C_*(K \vee r(S^2))$$
shows that the algebraic $2$-complex  $(F_*, \bd_*)$ is geometrically realizable after $r$-stabilization.
\end{proof}
\begin{corollary}[Wall]\label{cor:wall}\label{cor:algreal}
Every algebraic $2$-complex $F_*$ is chain homotopy equivalent to the chain complex $C_*(X)$ of a \textup{D2}-complex.
\end{corollary}
\begin{proof} The construction produces a chain homotopy equivalence 
$$  (F_*, \bd_*) \oplus (E_*, \bd_*)  \simeq C_*(K \vee r(S^2))$$
 after an $r$-stabilization of $F_*$, and in particular an isomorphism 
$ L \oplus E_2 = \pi_2(K)\oplus \La^r$, for some $r \geq 0$. Then one can  attach $3$-cells to $K\vee r(S^2)$, using the images in $\pi_2(K \vee r(S^2))$ of a free basis of the summand $E_2 \cong \Lambda^r$,  to produce a D2-complex $X$ and a chain homotopy equivalence $C(X) \simeq F_*$.
\end{proof}

\begin{remark}
The ingredients in the proof of Lemma \ref{lem:algreal} are  essentially the same as those used by  Wall to prove \cite[Theorem 4]{wall-finiteness2}. 
Similar ideas appear in  \cite[Appendix B]{johnson2}, \cite[Theorem 2.1]{Mannan:2009}.
\end{remark}

\begin{proof}[The proof of Theorem A]
Let $X$ be a finite $3$-complex which satisfies the D2-conditions. By Lemma \ref{lem:stableequiv}, there exists a  finite $2$-complex $K$ and  a simple homotopy equivalence $f\colon X' := X \vee r(S^2) \to K$, for any $r\geq b_3(X)$. Now let $G = \pi_1(X)$ be a finite group, and let $K_0$ denote  a minimal finite $2$-complex $K_0$ with fundamental group $G$. Then $\chi(K_0) = 1-\Def(G)$, and, after perhaps stabilizing further, we can assume that $K$ is simple homotopy equivalent to a stabilization of $K_0$. We then obtain a simple homotopy equivalence of the form
$$ X \vee r(S^2) \simeq K_0 \vee t(S^2) \vee r(S^2)$$
where $t \geq 0$ provided that $\chi(X) \geq 1-\Def(G) = \chi(K_0)$. We note that the arguments in \cite[\S 2]{hk4}  are at first completely algebraic (to obtain cancellation of the $\pi_2$ modules via elementary automorphisms), and then we show as above (compare the proof of \cite[Theorem B]{hk4}) how to realize the necessary elementary automorphisms by simple homotopy equivalences.

 If $\chi(X) > \chi(K_0)$, then $t \geq 1$ and we can construct simple self-equivalences of $K_0 \vee t(S^2) \vee r(S^2)$ to cancel the extra $r$ wedge summands of $X \vee r(S^2)$.  The resulting $2$-complex will be $K' \simeq K_0 \vee t(S^2)$.
 
  If $\chi(X) = \chi(K_0)$, then $t=0$ but we can perform the same operations after replacing $X$ by $X \vee S^2$, and the resulting $2$-complex will be $K' \simeq K_0 \vee S^2$.  In either case, the resulting $2$-complex $K'$ has non-minimal Euler characteristic $\chi(K') > \chi(K_0)$, so its simple homotopy type is uniquely determined by $G$ and $\chi(X)$ (see \cite[Theorem B]{hk4}).
\end{proof}

The techniques used in this proof also give a version for algebraic $2$-complexes (answering a question of Browning \cite[\S 5.6]{Browning:1978}). We recall that an \emph{$s$-basis} for a stably free $\La$-module $M$ is a free $\La$-basis for some stabilization $M \oplus \La^r$ by a free module.
\begin{corollary}\label{cor:browning}
 Let $F$ and $F'$ be $s$-based algebraic $2$-complexes over $\La= \ZG$, where $G$ is a finite group. If $\chi(F) = \chi(F') > \mu_2(G)$, then $F$ and $F'$ are simple chain homotopy equivalent.
\end{corollary}
\begin{proof}
We apply Corollary \ref{cor:algreal} and the method of proof for Theorem A.
\end{proof}

\begin{proof}[The proof of Theorem B]
The same  remarks as above apply to the proof of \cite[Theorem 2.1]{hk4}.  In addition, we note that $\mu_2(G) = 1-\Def(G)$ for all of the finite subgroups of $SO(3)$.  For these groups, $\Def(G) \geq -1$ (see Coxeter \cite[\S 6.4]{Coxeter:1980}), and $\mu_2(G)$ can be estimated by group cohomology using Swan \cite[Theorem 1.1]{Swan:1965}. We can now apply cancellation down to $r=0$ for fundamental groups which are finite subgroups of $SO(3)$. This proves that every algebraic $2$-complex with one of these fundamental groups is geometrically realizable.
\end{proof}

The uniform stability bound for D2-complexes in Theorem A is a special result for finite fundamental groups, based initially on the fact that their integral group rings are finite algebras over the integers. Here is a sample stability result which applies to certain infinite fundamental groups (compare Brown \cite{Brown:1981}).

\begin{proposition}\label{prop:noetherian} Let $G$ be a finitely presented group such that the integral group ring $\ZG$ is  noetherian of Krull dimension $d_G$. If $X$ is a finite complex with $\pi_1(X) = G$ satisfying the \textup{D2}-conditions, then $X \vee r(S^2)$ is simple homotopy equivalent to a finite $2$-complex, for $r \geq d_G + 1$, whose simple homotopy type is uniquely determined by $G$ and $\chi(X)$.
\end{proposition}
\begin{proof}(Sketch) The arguments follow the same outline as those used by Bass \cite[Chap IV.3.5]{Bass:1968}
to prove a cancellation theorem for modules using elementary automorphisms. The ingredients in these arguments were generalized to apply to non-commutative noetherian rings by Magurn, van der Kallen and Vaserstein \cite{Magurn:1988}, and Stafford \cite{Stafford:1977,Stafford:1990} (see also McConnell and Robson \cite[Chap.~11]{McConnell:2001}). The application to $2$-complexes  follows by realizing elementary automorphisms by simple homotopy self-equivalences, as in \cite[\S 2]{hk4}. 
\end{proof}

\begin{remark} For $G$ finite, the integral group ring $\ZG$ has Krull dimension $d_G=1$, so the Bass stability bound would be $d_G + 1 = 2$. If $G$ is a polycyclic-by-finite group, the group ring $\ZG$ is again noetherian and $d_G = h_G + 1 $, where $h_G$ denotes the \emph{Hirsch length} of $G$ (see \cite[6.6.1]{McConnell:2001}). The examples of  \cite{Dunwoody:1976},   \cite{Harlander:2006a}, \cite{Harlander:2006}, \cite{Harlander:2011} show that for general infinite fundamental groups (for example, the fundamental group of the trefoil knot), there can be (infinitely) many distinct $2$-complexes with the same Euler characteristic.
\end{remark}

\section{The relation gap problem}
We will conclude  by mentioning a related problem. If $F/R$ is a finite presentation for a group $G$, then the action of the free group $F$ by conjugation on the normal subgroup $R$ induces an action of $G$ on the quotient abelian group $R_{ab} := R/[R.R]$. This $\ZG$-module $R_{ab}$ is called the \emph{relation module} for $G$.

 Let $d(\Gamma)$ denote the minimum number of elements needed to generate a group $\Gamma$, and if a group $Q$ acts on $\Gamma$, then let $d_Q(\Gamma)$ denote the  minimum number of $Q$-orbits needed to generate $\Gamma$. Note that $d(\Gamma) \geq d_G(\Gamma)$.

In this notation, $d_F(R)$ is the minimum number of normal generators for $R$, and $d_G(R/[R.R])$ is the minimum number of $\ZG$-module generators for the module $R_{ab}$.

\begin{definition} For a finite presentation $F/R$ of a group $G$, the \emph{relation gap} is the difference $d_F(R) - d_G(R/[R,R])$. The \emph{relation gap problem} is to decide whether there exists a finitely presentation with a positive relation gap.
\end{definition}

The survey articles of Harlander \cite{Harlander:2000,Harlander:2015} provide some key examples (such as those constructed by Bridson and Tweedale \cite{Bridson:2007a}),  and a guide to the literature. 
A connection to the D2 problem is provided by the following result:

\begin{theorem}[{Dyer \cite[Theorem 3.5]{Harlander:2000}}] Let $G$ be a group with $H^3(G; \ZG) = 0$. If there exists a   finite presentation $F/R$ with a positive relation gap, realizing the deficiency of $G$,  then the \textup{D2} property does not hold for $G$.
\end{theorem}

The D2 problem can be considered a generalization of the Eilenberg-Ganea conjecture \cite{Eilenberg:1957}, which states that a group $G$ with cohomological dimension 2 also has geometric dimension 2. If $\text{cd}(G) = 2$ and  the classifying space $BG$ is homotopy equivalent to a finite complex, then
 $G$ will satisfy the Eilenberg-Ganea conjecture if  $G$ has the D2 property. 
 
 A striking result of Bestvina and Brady \cite[Theorem 8.7]{Bestvina:1997} shows that either the Eilenberg-Ganea conjecture is false, or there is a counterexample to the Whitehead conjecture, which states that every connected subcomplex of an aspherical 2-complex is aspherical.

%\bibliographystyle{ih}
%\bibliography{ihmain,d2}
%\end{document}
%%%%%%%%%%%%%%%%%%%%%%%%%%%%%
\providecommand{\bysame}{\leavevmode\hbox to3em{\hrulefill}\thinspace}
\providecommand{\MR}{\relax\ifhmode\unskip\space\fi MR }
\providecommand{\MRhref}[2]{%
  \href{http://www.ams.org/mathscinet-getitem?mr=#1}{#2}
}
\providecommand{\href}[2]{#2}

\end{document}